\documentclass[12pt,reqno]{amsart} 
\usepackage{amssymb,amscd,amsmath,amsthm,color}
\usepackage{amsmath}
\usepackage{amssymb}
\newtheorem{thm}{Theorem}
\newtheorem{lem}[thm]{Lemma}
\newtheorem{prop}[thm]{Proposition}
\newtheorem{cor}[thm]{Corollary}
\newcommand{\vertiii}[1]{{\left\vert\kern-0.25ex\left\vert\kern-0.25ex\left\vert
#1 \right\vert\kern-0.25ex\right\vert\kern-0.25ex\right\vert}}

\def \min   {\text {\rm min}}

\def \sup   {\text {\rm sup}}

\def \qed   {\hfill \vrule height6pt width 6pt depth 0pt}

\begin{document}

\title[]{Strong Convexity of Sandwiched Entropies and Related Optimization Problems}

\author[Rajendra Bhatia]{Rajendra Bhatia}
\address{Ashoka University, Sonepat\\ Haryana, 131029, India}

\email{rajendra.bhatia@ashoka.edu.in}

\author[Tanvi Jain]{Tanvi Jain}
\address{Indian Statistical Institute\\ New Delhi 110016, India}
\email{tanvi@isid.ac.in}

\author[Yongdo Lim]{Yongdo Lim}

\address{Department of Mathematics, Sungkyunkwan University\\ Suwon 440-746, Korea}

\email{ylim@skku.edu}


\date{August 17, 2018}

\begin{abstract}
We present several theorems on strict and strong convexity, and
higher order differential formulae for sandwiched quasi-relative
entropy (a parametrised version of the classical fidelity). These
are crucial for establishing global linear convergence of the
gradient projection algorithm for optimisation problems for these
functions. The case of the classical fidelity is of special interest
for the multimarginal optimal transport problem (the $n$-coupling
problem) for Gaussian measures.
\end{abstract}

\maketitle

\section{Introduction}
Let $\mathbb{P}$ be the space of  $n\times n$ complex positive
definite matrices. An element $A$ of  $\mathbb{P}$ with
${\mathrm{tr}} A=1$ is called a \emph{density matrix} or a
\emph{state}. Many of the statements in this paper are of special
interest for density matrices though we do not make that
restriction. The \emph{fidelity} between two elements $A$ and $B$ of
$\mathbb{P}$ is defined by
\begin{eqnarray}\label{E:Fidelity}
F(A,B)={\mathrm{tr}}
\left(A^{\frac{1}{2}}BA^{\frac{1}{2}}\right)^{\frac{1}{2}}.
\end{eqnarray}
 Fidelity plays an important role
in quantum information theory and quantum computation, and it has
deep connections with quantum entanglement, quantum chaos, and
quantum phase transitions. See \cite{u,u2}. Although fidelity by
itself is not a metric, it has played a role as a measure of the
��closeness�� of two states. It occurs also in another
context. There is a metric on $\mathbb{P}$ defined as
\begin{eqnarray}
d(A,B)=\left[\frac{{\mathrm{tr}}(A+B)}{2}-{\mathrm{tr}}\left(A^{\frac{1}{2}}BA^{\frac{1}{2}}\right)^{\frac{1}{2}}\right]^{\frac{1}{2}}
\end{eqnarray}
which is called the \emph{Bures distance} in the literature on
quantum information and the \emph{Wasserstein metric} in statistics
and the theory of optimal transport. See \cite{bur,dl,gs1,ks,op}.

The multimarginal optimal transport problem (alternatively, the
coupling problem) involves solving the minimization problem: given
$A_{1},\dots, A_{m}$ in ${\Bbb P}$ and weights $w_{1},\dots,w_{m},$
find
\begin{eqnarray}\label{E:Opt}
 \underset{X\in {\Bbb P}}{\min}\,\,\, \sum_{j=1}^{m}w_{j}
 d^{2}(X,A_{j}).
\end{eqnarray}
This minimization problem  coincides with the least squares problem
of Gaussian measures
 for the Wasserstein distance
between probability measures with finite second moment on ${\Bbb
R}^{n}.$ See \cite{ac,gm,gs1,ks,op,t}. The concavity and strict
concavity of the function
\begin{eqnarray}
f(X)={\mathrm{tr}}\left(A^{\frac{1}{2}}XA^{\frac{1}{2}}\right)^{\frac{1}{2}}
\end{eqnarray}
on ${\Bbb P}$ play a very crucial role in the proofs of existence
and uniqueness of the solution to (\ref{E:Opt}). See \cite{bjl}.

In some recent works a parameterized version of fidelity defined as
\begin{eqnarray}
F_{t}(A,B)={\mathrm{tr}}\left(A^{\frac{1-t}{2t}}BA^{\frac{1-t}{2t}}\right)^{t},
\ \ \ \ t\in (0,\infty)
\end{eqnarray}
has been studied. See \cite{fl,wwy}. The usual fidelity
(\ref{E:Fidelity}) is the special case $t=1/2.$ In \cite{wwy}
$F_{t}(A,B)$ is called the \emph{sandwiched quasi-relative entropy}.
Using this the \emph{sandwiched R\'enyi relative entropy} is defined
as
\begin{eqnarray}\label{E:ML}
D_{t}(B\parallel A)=\frac{1}{t-1}\log F_{t}(A,B), \ \ \ t\in
(0,\infty)\setminus \{1\}.
\end{eqnarray}
This is a variant of the traditional \emph{relative R\'enyi entropy}
defined as
\begin{eqnarray}
D'_{t}(B\parallel A)=\frac{1}{t-1}\log {\mathrm{tr}}
\left(A^{1-t}B^{t}\right).
\end{eqnarray}
Among other things, it is known \cite{mds} that
\begin{eqnarray}\label{E:Thomp}
\lim_{t\to \infty} D_{t}(B\parallel A)=\|\log
A^{-\frac{1}{2}}BA^{-\frac{1}{2}}\|
\end{eqnarray}
 and
\begin{eqnarray}\label{E:RE}
\lim_{t\to 1} D_{t}(B\parallel A)=\frac{1}{{\mathrm{tr}}
B}{\mathrm{tr}}\left[B(\log B-\log A)\right],
\end{eqnarray}
where $||\cdot||$ is the operator norm
$${\Vert A \Vert}=\underset{\Vert x\Vert=1}{\sup}\Vert Ax\Vert,$$
which for a positive semidefinite matrix $A$ is equal to $\lambda_1(A),$ the largest
eigenvalue of $A.$  It turns out \cite{ACS} that the expression in
(\ref{E:Thomp}) coincides with
 $$d_T(A,B):=\max\{\log \lambda_1(AB^{-1}), \log\lambda_1(BA^{-1})\}$$
 and is closely  related to the max-relative entropy $D_{\max}(A\Vert B):=\log \lambda_{1}(AB^{-1})$ in the context of quantum
 information theory \cite{Datta}.
We note that $d_T$ is known as  the \emph{Thompson metric} on ${\Bbb
P}$ and  is a complete metric invariant under inversion and
congruence transformations \cite{th,nu}, and the expression in
(\ref{E:RE}) is the \emph{relative entropy}, first introduced by
Umegaki.

The entity (\ref{E:ML}) was introduced by M\"uller-Lennert et al in
\cite{mds} and by Wilde et al in \cite{wwy}.  Several of its
properties were established in these papers and some others
conjectured. Since then these have been established in various
papers. In particular, we draw attention to the paper \cite{fl} by
Frank and Lieb. In \cite{wwy} Wilde, Winter and Yang have employed
$D_{t}(A\parallel B)$ to prove theorems on the capacity of
entanglement-breaking channels. Differentiability, monotonicity and
convexity properties of $F_{t}$ and $D_{t}$ are a major theme in all
these papers.

In this paper we study some related, though slightly different,
convexity problems. Let $f:{\Bbb P}\to {\Bbb R}$ be a smooth
function. Let $\nabla f(X)$ and $\nabla^2f(X)$ denote the gradient
and the Hessian of $f.$ See \cite{bv} for gradient and Hessian of
scalar valued functions. Suppose $f$ is strictly convex. The
\emph{Bregman distance} associated with $f$ is the function
${\mathcal D}_{f}:{\Bbb P}\times {\Bbb P}\to {\Bbb R}$ defined as
\begin{eqnarray}
{\mathcal D}_{f}(Y,X)=f(Y)-f(X)-\langle \nabla f(X), Y-X\rangle,\label{eq11}
\end{eqnarray}
where $\langle X,Y\rangle={\mathrm{tr}}(XY)$ on ${\Bbb H},$ the
space of $n\times n$ complex Hermitian matrices. The convexity of
$f$ ensures that ${\mathcal D}_{f}(Y,X)\geq 0,$ and strict convexity
ensures that it is zero if and only if $X=Y.$ Let ${\Bbb K}$ be a
compact convex subset of ${\Bbb P}$. We say that $f$ is
\emph{$k$-strongly convex} on ${\Bbb K}$ (with $k>0$) if for all
$X,Y\in {\Bbb K}$
\begin{eqnarray}\label{E:S-convexity}
{\mathcal D}_{f}(Y,X)\geq \frac{k}{2}\|X-Y\|_{2}^{2}.
\end{eqnarray} Here
$\|A\|_{2}=\left({\mathrm{tr}}A^*A\right)^{\frac{1}{2}}$ is the
Hilbert-Schmidt norm. The condition (\ref{E:S-convexity}) says
\begin{eqnarray}
f(Y)\geq f(X)+\langle \nabla f(X), Y-X\rangle
+\frac{k}{2}\|X-Y\|_{2}^{2}.
\end{eqnarray}
So $f$ is $k$-strongly convex on ${\Bbb K}$ if and only if
\begin{eqnarray}
\nabla^2f(X)\geq kI,\label{eq13}
\end{eqnarray}
for all $X\in {\Bbb K}.$ On the other hand, we say that $f$ is
\emph{$k$-smooth} on ${\Bbb K}$ if $\nabla f$ is $k$-Lipschitz;
i.e., \begin{eqnarray} \Vert\nabla f(X)-\nabla f(Y)\Vert_{2}\leq
k\Vert X-Y\Vert_{2},
\end{eqnarray}
for all $X,Y\in {\Bbb K}.$ This condition is equivalent to
\begin{eqnarray}
\nabla^2 f(X)\leq k\ I,\label{eq15}
\end{eqnarray}
for all $X\in {\Bbb K}.$

The two constants $k$ in \eqref{eq13} and \eqref{eq15} play
a fundamental role in the design and convergence analysis of
optimisation algorithms. We refer the reader to Chapter 9 of the
standard text \cite{bv}. Here it is also pointed out that these
constants $``$are known only in rare cases''. The main new  result
in this paper is the following.

\begin{thm}\label{T:main}
Let $f:{\Bbb P}\to {\Bbb R}_+$ be the function
\begin{eqnarray}\label{E:mainf}
f(X)={\mathrm{tr}}\left(A^{\frac{1-t}{2t}}XA^{\frac{1-t}{2t}}\right)^{t}
\end{eqnarray}
where $A\in {\Bbb P}$ and $0<t<1.$ Let ${\Bbb K}$ be a compact
convex subset of ${\Bbb P}.$  Let $\alpha, \beta$ be positive
numbers such that $\alpha I\leq Y\leq \beta I$ for all $Y\in {\Bbb
K}\cup \{A\}.$ Then for all $X\in {\Bbb K}$
\begin{eqnarray}\label{E:convexity}
t(1-t)\alpha^{1-t}\beta^{t-2}\leq -\nabla^2f(X)\leq
t(1-t)\beta^{1-t}\alpha^{t-2}.
\end{eqnarray}
\end{thm}

In other words, the function $-f$ is $k_{1}$-strongly convex and
$k_{2}$-smooth on ${\Bbb K}$ with $k_{1}, k_{2}$ given by the two
extreme sides of (\ref{E:convexity}). The condition number of an
operator $A$ is defined as
$${\mathrm{cond}}(A)=\Vert A\Vert  \Vert A^{-1}\Vert.$$ As a corollary to Theorem
\ref{T:main} we have:

\begin{cor} Let $f$ be the function defined in $(\ref{E:mainf}).$
Then for all $X\in {\Bbb K}$
$${\mathrm{cond}}\left(\nabla^2f(X)\right)\leq
\left(\frac{\beta}{\alpha}\right)^{3-2t}.$$
\end{cor}

Now suppose $A_{j}, 1\leq j\leq m$ are positive definite matrices,
and let $\alpha I\leq A_{j}\leq \beta I$ for all $j.$ It is known
\cite{ac} that the minimization problem (\ref{E:Opt}) has a unique
solution $X$ and $\alpha I \leq X\leq \beta I.$ The objective
function in (\ref{E:Opt}) is
$$\varphi(X)=\sum_{j=1}^{m}w_{j}\left[\frac{{\mathrm{tr}}(A_{j}+X)}{2}-{\mathrm{tr}}\left({A_j}^{\frac{1}{2}}X{A_j}^{\frac{1}{2}}\right)^{\frac{1}{2}}\right].$$
The first term in the square brackets above is linear in $X,$ and
its second derivative is zero. Our theorem shows that
$$\frac{1}{4}\frac{\alpha^{1/2}}{\beta^{3/2}}\leq
\nabla^2\varphi(X)\leq
\frac{1}{4}\frac{\beta^{1/2}}{\alpha^{3/2}}.$$ The condition number
of $\nabla^2\varphi(X)$ is bounded by
$\left(\frac{\beta}{\alpha}\right)^{2}.$
 We generalize this result into the setting of
 sandwiched quasi-relative entropy $F_t(A,B), 0<t<1.$
Let
$$\varphi_{t}(X)=\sum_{j=1}^{m}w_{j}\left[{\mathrm{tr}}((1-t)A_{j}+tX)-{\mathrm{tr}}\left({A_j}^{\frac{1-t}{2t}}X{A_j}^{\frac{1-t}{2t}}\right)^{t}\right].$$

\begin{cor}\label{C:main2} The function $\varphi_{t}:{\Bbb P}\to {\Bbb R}_{+}$
is  strictly convex  and has a unique minimizer. Moreover, it is
$t(1-t)\beta^{1-t}\alpha^{t-2}$-smooth and
$t(1-t)\beta^{t-2}\alpha^{1-t}$-strongly convex.
\end{cor}

Theorem $1$ is about second order derivatives of the fidelity
function. The classical fidelity case is t = 1/2, and the results
are new even for that case. Our methods lead to several interesting
observations for the first and higher order derivatives as well.
These are of independent interest and are given in Section 2 of the
paper. Section 3 includes a proof of Theorem 1. A proof of Corollary
\ref{C:main2} and the standard gradient projection method where this
can be put to use are obtained in Section 4.

\section{Derivative Computations}

Let $f$ be a smooth map from ${\Bbb P}$ into the positive half-line
${\Bbb R}_{+}=[0,\infty).$ We denote by $Df(X)$ the (Fr\'echet)
derivative of $f$ at $X,$ and by $\nabla f(X)$ the gradient of $f$
at $X.$  $Df(X)$ is a linear map from the space ${\Bbb H}$ of
$n\times n$ Hermitian matrices into ${\Bbb R},$ and its action is
given by
$$Df(X)(Y)=\frac{d}{dt}\Big|_{t=0}f(X+tY).$$
$\nabla f(X)$ is an element of ${\Bbb H}$ and is related to $Df(X)$
by the equation
$$Df(X)(Y)=\langle \nabla f(X),Y\rangle={\mathrm{tr}}(\nabla f(X)Y).$$

Of interest here are special kinds of functions. Let $f$ be a smooth
map from ${\Bbb R}_{+}$ into itself and let $f$ also denote the map
this induces from ${\Bbb P}$ into itself.  Let ${\hat
f}(A)={\mathrm{tr}} f(A).$ As expected, convexity properties of $f$
are inherited by ${\hat f}.$ In some situations it may be useful to
consider functions other than the trace. Let $\Phi$ be a symmetric
gauge function on ${\Bbb R}^{n},$  i.e., a norm on ${\Bbb R}^n$
which is invariant under sign changes and permutations of the
components, and let $\|\cdot\|_{\Phi}$ be the corresponding
unitarily invariant norm on the space ${\Bbb M}(n)$ of $n\times n$
matrices. See Chapter IV of \cite{rbh}. If
$s(A)=(s_{1}(A),\dots,s_{n}(A))$ is the $n$-tuple of singular values
of $A,$ then
$$\|A\|_{\Phi}=\Phi(s(A))=\Phi(s_{1}(A),\dots,s_{n}(A)).$$

Every symmetric gauge function is \emph{monotone}; i.e., if $x$ and
$y$ are two vectors with $0\leq x\leq y$ for all $j,$ then
$\Phi(x)\leq \Phi(y).$ We say that $\Phi$ is \emph{strictly
monotone} if $\Phi(x)<\Phi(y)$ whenever $0\leq x_j\leq y_j$ for all
$j$ and $x_j<y_j$ for at least one $j.$ For example, the symmetric
gauge functions $\Phi(x)=\left(\sum_{j=1}^n|x_j|^p\right)^{1/p}$ are
strictly monotone for $1\leq p<\infty.$

Let $x = (x_1, \ldots, x_n)$ and $y= (y_1, \ldots, y_n)$ be two
$n$-tuples of nonnegative numbers. Let $x_1^{\downarrow} \geq
x_2^{\downarrow} \geq \ldots \geq x_n^{\downarrow}$  be the
decreasing rearrangement of $x_{1},\dots,x_{n}.$  If for all $1\leq
k\leq n$
\begin{equation*}
\sum^{k}_{j=1} x_j^{\downarrow} \,\,\, \leq
\sum^{k}_{j=1}y_j^{\downarrow},\,\,\, \label{eq12}
\end{equation*}
we say that $x$ is  \emph{weakly majorised} by $y.$
 If, in addition to (\ref{eq12}) we also have
\begin{equation*}
\sum^{n}_{j=1} x_j^{\downarrow} \,\,\, =
\sum^{n}_{j=1}y_j^{\downarrow},\,\,\, \label{eq14}
\end{equation*}
 we say $x$ is
\emph{majorised} by $y,$ and write this as $ x\prec y.$ See Chapter
II of \cite{rbh} for facts on majorization need here.

\begin{lem}\label{r}
Let $x,y$ be two vectors with nonnegative coordinates that are not
permutations of each other. Suppose $x\prec y$. Then for every
strictly convex function $f$ on nonnegative reals and every strictly
monotone symmetric gauge function $\Phi,$ we have
$$\Phi(f(x_1),\dots,f(x_n))<\Phi(f(y_1),\dots,f(y_n)).$$
\end{lem}

\begin{proof} If $x\prec y$, then $x$ can be expressed as a convex
combination
$$x=\sum a_{\sigma}y_{\sigma},$$
where $\sigma$ varies over all permutations on $n$ symbols, and
$y_{\sigma}$ denotes the vector $(y_{\sigma(1)},\dots,
y_{\sigma(n)}).$ If $x$ and $y$ are not permutations of each other,
there are at least two distinct terms in this convex combination.
Since $f$ is convex, $$f(x_j)\leq \sum a_{\sigma}f(y_{\sigma(j)})$$
for all $j.$ If $f$ is strictly convex, then this inequality is
strict for some $j.$ The statement of the lemma then follows from
the properties of $\Phi.$
\end{proof}

\begin{thm}
Let $f$ be a function from ${\Bbb R}_{+}$ into itself, and let
$\Vert\cdot\Vert_{\Phi}$ be a unitarily invariant norm on ${\Bbb
M}(n).$ Let ${\hat f}_{\Phi}$ be the map from ${\Bbb P}$ into ${\Bbb
R}_{+}$ defined by
$${\hat f}_{\Phi}(A)=\|f(A)\|_{\Phi}.$$ If $f$ is convex,
then so is ${\hat f}_{\Phi}.$ Further, if $f$ is strictly convex and
$\Phi$ is strictly monotone, then ${\hat f}_\Phi$ is strictly
convex.
\end{thm}

\begin{proof}
Let $A,B\in {\Bbb P},$ and let $C=(1/2)(A+B).$ Let
$\{\lambda_{j}(C)\}$ denote the decreasingly ordered eigenvalues of
$C,$ and let $\{u_{j}\}$ be the corresponding orthonormal set of
eigenvectors. Then
\begin{eqnarray*}
\|f(C)\|_{\Phi}&=&\Phi\left(\lambda_{1}(f(C)),\dots,\lambda_{n}(f(C))\right)\\
&=&\Phi\left(f(\lambda_{1}(C)),\dots,f(\lambda_{n}(C))\right)\\
&=&\Phi\left(f(\langle u_{1},Cu_{1}\rangle), \dots, f(\langle
u_{n},Cu_{n}\rangle)\right).
\end{eqnarray*} Since $f$ is convex,
\begin{align}
f\left(\langle u_{j}, Cu_{j}\rangle\right) \nonumber
&=f\left(\frac{\langle
u_{j},Au_{j}\rangle+\langle u_{j}, Bu_{j}\rangle}{2}\right)\\
&\leq\frac{1}{2}\left[f(\langle u_{j},Au_{j}\rangle)+f(\langle
u_{j},Bu_{j}\rangle)\right].\label{rf}
\end{align}
Every symmetric gauge function is monotone and convex. So, the
relations above give
\begin{eqnarray}
\Vert f(C)\Vert_{\Phi}&\leq&\frac{1}{2}\Phi\left(f(\langle
u_{1},Au_{1}\rangle),\dots, f(\langle u_{n},Au_{n}\rangle)\right)\nonumber\\
&{}&+\frac{1}{2}\Phi\left(f(\langle u_{1},Bu_{1}\rangle),\dots,
f(\langle u_{n},Bu_{n}\rangle)\right).\label{eqnew}
\end{eqnarray}
Since $f$ is convex, by Problem IX. 8. 14 in \cite{rbh} we see that
$$f\left(\langle u_{j},Au_{j}\rangle\right)\leq \langle
u_{j},f(A)u_{j}\rangle.$$ By the Schur majorisation theorem
(Exercise II. 1.2 in \cite{rbh}) the $n$-tuple $\{\langle
u_{j},f(A)u_{j}\rangle\}$ is majorised by the eigenvalue $n$-tuple
$\{\lambda_{j}(f(A))\}.$ Every symmetric gauge function is monotone
with respect to majorisation ($``$isotone" in the terminology used
on page 41 of \cite{rbh}). Combining these observations we see that
\begin{eqnarray*}
\Phi\left(f(\langle u_{1},Au_{1}\rangle),\dots, f(\langle
u_{n},Au_{n}\rangle\right)
&\leq&\Phi\left(\lambda_{1}(f(A)),\dots,\lambda_{n}(f(A))\right)\\
&=&\|f(A)\|_{\Phi}.
\end{eqnarray*}
The same argument applies to $B$ in place of $A$. Hence
\begin{eqnarray}\label{eqneww}\Vert f(C)\Vert_{\Phi}\leq \frac{1}{2}\Vert
f(A)\Vert_{\Phi}+\frac{1}{2}\Vert f(B)\Vert_{\Phi}. \end{eqnarray}
This shows that ${\hat f}_{\Phi}$ is convex if $f$ is convex. Now,
suppose  $f$ is strictly convex and $\Phi$ is strictly monotone. Let
$A\ne B.$ There are two possibilities: (i) There exists a  $j$ such
that $\langle u_j,Au_j\rangle \neq \langle u_j,Bu_j\rangle.$ Then
for this $j,$ the inequality (\ref{rf}) is strict and hence the
inequality (\ref{eqnew}) is also strict. The argument above then
shows the inequality (\ref{eqneww}) is strict. (ii) If  $\langle
u_j,Au_j\rangle = \langle u_j,Bu_j\rangle$ for all $j,$ then in the
orthonormal basis $\{u_1,\dots,u_n\},$ $C$ is diagonal, and the
diagonals of $A$ and $B$ are equal. This means that
${\mathrm{diag}}(C)={\mathrm{diag}}(A)={\mathrm{diag}}(B).$ Since
$A\neq B,$ neither $A$ nor $B$ is diagonal. By Schur's majorization
  theorem (See (II. 14) of \cite{rbh})
$${\mathrm{diag}}(A)\prec \lambda(A),$$
where $\lambda(A)$ is the vector whose components are the
eigenvalues of $A$. Since $A$ is not diagonal, ${\mathrm{diag}}(A)$
is not a permutation of $\lambda(A)$ (because $||A||_2 =
||\lambda(A)||_2$).  It follows from Lemma \ref{r} that
$$\Vert f\left({\mathrm{diag}(A)}  \right)\Vert_{\Phi}<\Vert f(A)\Vert_\Phi.$$
The same argument applies to $B.$ Since $C={\mathrm{diag}}(A),$ this
shows the inequality (\ref{eqneww}) is strict. This proves the last
statement of the theorem.
\end{proof}

The sum of singular values is a strictly monotone unitarily
invariant norm. So, the function
$${\hat f}(A)={\mathrm{tr}} f(A)$$
is (strictly) convex if $f$ is (strictly) convex. In addition, using
the linearity of the trace function we can see that ${\hat f}(A)$ is
(strictly) concave if $f$ is (strictly) concave. This is a
well-known fact. See \cite{c}.

\begin{cor}\label{E:SC}
The function $f(X)={\mathrm{tr}} X^t$ on positive definite matrices
is strictly concave if $0<t<1$ and strictly convex if $1<t<\infty,$
or if $t<0.$
\end{cor}

\begin{lem}\label{L:diff}
Let $f$ be a smooth function on ${\Bbb R}_{+}$ and let ${\hat f}$ be
the function on ${\Bbb P}$ defined as ${\hat f}(X)={\mathrm{tr}}
f(X).$ Then for all $X\in {\Bbb P}$ and $Y\in {\Bbb H},$
$$D{\hat f}(X)(Y)={\mathrm{tr}}\left(f'(X)Y\right).$$
\end{lem}

\begin{proof}
Let $\lambda_{1},\dots,\lambda_{n}$ be the eigenvalues of $X$ and
let $L_{f}(X)$ be the Loewner matrix
$$L_{f}(X)=\left[\frac{f(\lambda_{i})-f(\lambda_{j})}{\lambda_{i}-\lambda_{j}}\right].$$
The difference quotient in this expression is the $ij$th entry of
$L_{f}(X),$ and it is understood that this is equal to
$f'(\lambda_{i})$ if $\lambda_{i}=\lambda_{j}$. By the
Daleckii-Krein formula (Theorem V. 3.3 in \cite{rbh}) the derivative
$Df(X)$ is given by
$$Df(X)(Y)=L_{f}(X)\circ Y,$$
where $\circ$ stands for the Hadamard product (entrywise product) of
two matrices taken in an orthonormal basis in which $X$ is diagonal.
Combining this with the linear functional ${\mathrm{tr}},$ we get
$$D{\hat f}(X)(Y)={\mathrm{tr}} \left(L_{f}(X)\circ
Y\right)={\mathrm{tr}}\left(f'(X)Y\right).$$
\end{proof}

To state the next proposition we need the notion of the weighted
geometric mean of two positive definite matrices. This is defined as
\begin{eqnarray}\label{E:GM}A\#_{t}B=A^{\frac{1}{2}}(A^{-\frac{1}{2}}BA^{-\frac{1}{2}})^tA^{\frac{1}{2}},
\ \ 0\leq t\leq 1. \end{eqnarray}
This is a smooth curve joining $A$
and $B,$ and is a geodesic with respect to the Riemannian distance
$$\delta(A,B)=\Vert\log A^{-\frac{1}{2}}BA^{-\frac{1}{2}}\Vert_{2},$$
on ${\Bbb P}.$ See Chapter 6 of \cite{rbh1}. The right hand side of
(\ref{E:GM}) is meaningful for all $t\in {\Bbb R},$  and we continue
to use the notation $A\#_{t}B$ for it.

\begin{prop}\label{P:Der}
Let $A$ be any element of ${\Bbb P}$ and let $t\in {\Bbb R}.$ Let
$h:{\Bbb P}\to {\Bbb R}_{+}$ be the map $h(X)={\mathrm{tr}}
\left(A^{\frac{1}{2}}XA^{\frac{1}{2}}\right)^{t}.$ Then
\begin{eqnarray}Dh(X)(Y)=t\ {\mathrm{tr}}\left(A\#_{1-t}X^{-1}\right)Y;
\end{eqnarray}
i.e.,
\begin{eqnarray}
\nabla h(X)=t\left(A\#_{1-t}X^{-1}\right).
\end{eqnarray}
\end{prop}
\begin{proof}
Let $k(X)={\mathrm{tr}} X^{t}.$ Then by Lemma \ref{L:diff},
$Dk(X)(Y)=t {\mathrm{tr}} X^{t-1}Y.$ By the chain rule
\begin{eqnarray*}
Dh(X)(Y)&=&Dk(A^{\frac{1}{2}}XA^{\frac{1}{2}})(A^{\frac{1}{2}}YA^{\frac{1}{2}})\\
&=&t\
{\mathrm{tr}}(A^{\frac{1}{2}}XA^{\frac{1}{2}})^{t-1}A^{\frac{1}{2}}YA^{\frac{1}{2}}\\
&=&t\
{\mathrm{tr}}A^{\frac{1}{2}}\left(A^{-\frac{1}{2}}X^{-1}A^{-\frac{1}{2}}\right)^{1-t}A^{\frac{1}{2}}Y\\
&=&t\ {\mathrm{tr}}\left(A\#_{1-t}X^{-1}\right)Y.
\end{eqnarray*}
\end{proof}

Extremal representations for the fidelity $F(A,B)$ are useful in
deriving various relations. See \cite{op,bjl}. Our next theorem
gives such representations for $F_{t}(A,B).$ Some of these have been
derived in \cite{fl} and \cite{bft}.

\begin{thm}\label{T:Ineq}
Let $A,B$ be any two elements of ${\Bbb P}$ and let $0<t<1.$ Then
\begin{itemize}
\item[(i)] $F_{t}(A,B)=\underset{X\in {\Bbb P}}{\min}\,\,\,
{\mathrm{tr}}\left[(1-t)\left(A^{\frac{t-1}{2t}}XA^{\frac{t-1}{2t}}\right)^{\frac{t}{t-1}}+tXB\right].$
\item[(ii)] $F_{t}(A,B)=\underset{X\in {\Bbb P}}{\min}\,\,\,
\left[{\mathrm{tr}}(A^{\frac{t-1}{2t}}XA^{\frac{t-1}{2t}})^{\frac{t}{t-1}}\right]^{1-t}
\left[{\mathrm{tr}} XB\right]^t.$
\item[(iii)] $F_{t}(A,B)=\underset{X\in {\Bbb P}}{\min}\,\,\,
{\mathrm{tr}}\left[tA^{\frac{1-t}{t}}X+(1-t)\left(B^{-\frac{1}{2}}XB^{-\frac{1}{2}}\right)^{\frac{t}{t-1}}\right].$
\item[(iv)] $F_{t}(A,B)=\underset{X\in {\Bbb P}}{\min}\,\,\,
\left[{\mathrm{tr}}A^{\frac{1-t}{t}}X\right]^t\left[{\mathrm{tr}}(B^{-\frac{1}{2}}XB^{-\frac{1}{2}})^{\frac{t}{t-1}}\right]^{1-t}.$
\end{itemize}
\end{thm}

\begin{proof}
The representations (i) and (ii) have been derived and used in \cite{fl}. We will give here proofs of (iii) and (iv). The same ideas can be used to give proofs of (i) and (ii), which are different from the ones given in \cite{fl}.

 (iii) By Corollary \ref{E:SC}, the function
$$f(X)={\mathrm{tr}}\left[tA^{\frac{1-t}{t}}X+(1-t)\left(B^{-\frac{1}{2}}XB^{-\frac{1}{2}}\right)^{\frac{t}{t-1}}\right]$$
is strictly convex for $0<t<1.$ Using Proposition \ref{P:Der} we see
that $$\nabla
f(X)=t\left(A^{\frac{1-t}{t}}-B^{-1}\#_{\frac{1}{1-t}}X^{-1}\right).$$ So
$\nabla f(X_{0})=0$ if and only if
\begin{eqnarray}\label{E:GE}A^{\frac{1-t}{t}}=B^{-1}\#_{\frac{1}{1-t}}X_{0}^{-1}.
\end{eqnarray}
Now, if $C=Y^{-1}\#_{\alpha}X^{-1},$ then from the definition
(\ref{E:GM}) one can see that $X=Y\#_{\frac{1}{\alpha}}C^{-1}.$ So,
from (\ref{E:GE}) we see that $\nabla f(X_{0})=0$ if and only if
\begin{eqnarray*}
X_{0}=B\#_{1-t}A^{\frac{t-1}{t}}=A^{\frac{t-1}{t}}\#_{t}B.
\end{eqnarray*}
A little calculation shows that
$${\mathrm{tr}}A^{\frac{1-t}{t}}X_{0}={\mathrm{tr}}\left(B^{-\frac{1}{2}}X_0B^{-\frac{1}{2}}\right)^{\frac{t}{t-1}}=
{\mathrm{tr}}\left(B^{\frac{1}{2}}A^{\frac{1-t}{t}}B^{\frac{1}{2}}\right)^{t}.$$
For any two positive matrices $P$ and $Q,$
$${\mathrm{tr}}\left(P^{\frac{1}{2}}QP^{\frac{1}{2}}\right)^{t} = {\mathrm{tr}}Q^{\frac{1}{2}}P^{\frac{1}{2}}\left(P^{\frac{1}{2}}QP^{\frac{1}{2}}\right)^{t}P^{-\frac{1}{2}}Q^{-\frac{1}{2}}={\mathrm{tr}}\left(Q^{\frac{1}{2}}PQ^{\frac{1}{2}}\right)^{t}.$$
 Hence
$${\mathrm{tr}}A^{\frac{1-t}{t}}X_{0}={\mathrm{tr}}\left(B^{-\frac{1}{2}}X_0B^{-\frac{1}{2}}\right)^{\frac{t}{t-1}}={\mathrm{tr}}\left(A^{\frac{1-t}{2t}}BA^{\frac{1-t}{2t}}\right)^{t}=F_{t}(A,B).$$
We have shown that $X_{0}$ is the unique minimizer for the problem
(iii) and the minimum value is equal to $F_{t}(A,B).$

(iv) For an $n\times n$ matrix  $X$, let $|X|$be the absolute value
of $X$ defined as $|X|=(X^{*}X)^{\frac{1}{2}}.$ Let $p,q,r$ be
positive numbers with $\frac{1}{p}+\frac{1}{q}=\frac{1}{r}.$ By the
matrix version of H\"older's inequality (Exercise IV. 2.7 in
\cite{rbh})
$${\mathrm{tr}}|ST|^{r}\leq
\left({\mathrm{tr}}|S|^{p}\right)^{\frac{r}{p}}\left({\mathrm{tr}}|T|^{q}\right)^{\frac{r}{q}},$$
Note that
\begin{eqnarray*}
F_{t}(A,B)&=&{\mathrm{tr}}\left(A^{\frac{1-t}{2t}}BA^{\frac{1-t}{2t}}\right)^{t}\\
&=&{\mathrm{tr}}\left(B^{\frac{1}{2}}A^{\frac{1-t}{t}}B^{\frac{1}{2}}\right)^{t}\\
&=&{\mathrm{tr}}
\left(B^{\frac{1}{2}}X^{-\frac{1}{2}}X^{\frac{1}{2}}A^{\frac{1-t}{2t}}A^{\frac{1-t}{2t}}X^{\frac{1}{2}}X^{-\frac{1}{2}}
B^{\frac{1}{2}}\right)^{t}.
\end{eqnarray*}
Taking $S=A^{\frac{1-t}{2t}}X^{\frac{1}{2}},$ $T=X^{-\frac{1}{2}}B^{\frac{1}{2}},$ $p=1$ and $q=\frac{t}{1-t}$
in H\"older's inequality we get
\begin{eqnarray*}
F_t(A,B)
&\leq&
\left[{\mathrm{tr}}\left(X^{\frac{1}{2}}A^{\frac{1-t}{t}}X^{\frac{1}{2}}\right)\right]^{t}
\left[{\mathrm{tr}}\left(B^{\frac{1}{2}}X^{-1}B^{\frac{1}{2}}\right)^{\frac{t}{1-t}}\right]^{1-t}\\
&=&\left[{\mathrm{tr}}A^{\frac{1-t}{t}}X\right]^{t}
\left[{\mathrm{tr}} \left(B^{-\frac{1}{2}}XB^{-\frac{1}{2}}\right)^{\frac{t}{t-1}}\right]^{1-t}.
\end{eqnarray*}
We have seen in the proof of (iii) that when
$X=X_{0}=A^{\frac{t-1}{t}}\#_{t}B,$ then each of the
expressions inside the square brackets on the right hand side is equal to
$F_{t}(A,B).$ This proves (iv).
\end{proof}

For given $A,B\in {\Bbb P}$ let
\begin{eqnarray}
\gamma(t)=\left(A^{\frac{1-t}{2t}}BA^{\frac{1-t}{2t}}\right)^{t} \ \
\ \ 0\leq t\leq 1.
\end{eqnarray} The value $\gamma(0)$ is given by the following
proposition, first established in \cite{ad}.

\begin{prop} For all $A,B\in {\Bbb P}$
we have
$$\lim_{t\to 0^{+}} \left(A^{\frac{1-t}{2t}}BA^{\frac{1-t}{2t}}\right)^{t}
=A.$$
\end{prop}

\begin{proof} Let $\alpha,\beta$ be positive numbers such that
$\alpha I\leq B\leq \beta I.$ Then
$$\alpha A^{\frac{1-t}{t}}\leq
A^{\frac{1-t}{2t}}BA^{\frac{1-t}{2t}}\leq \beta A^{\frac{1-t}{t}},$$
and hence for $0<t<1,$
$$\alpha^{t}A^{1-t}\leq
\left(A^{\frac{1-t}{2t}}BA^{\frac{1-t}{2t}}\right)^{t}\leq
\beta^{t}A^{1-t}.$$ Taking the limit as $t\to 0,$ we see that
$$A\leq \lim_{t\to
0^{+}}\left(A^{\frac{1-t}{2t}}BA^{\frac{1-t}{2t}}\right)^{t}\leq
A.$$ This proves the proposition.
\end{proof}

Thus $\gamma(t), 0\leq t\leq 1,$ is a differentiable curve joining
$A$ and $B.$ It is of interest to compare this with two other
curves: the Riemannian geodesic (\ref{E:GM}) and the straight line
segment. In this direction we have

\begin{thm} For $0<t<1$
\begin{eqnarray}\label{E:Tracein}
{\mathrm{tr}}A\#_{t}B\leq {\mathrm{tr}} A^{1-t}B^t\leq
{\mathrm{tr}}\left(A^{\frac{1-t}{2t}}BA^{\frac{1-t}{2t}}\right)^{t}\leq
{\mathrm{tr}}\left[(1-t)A+tB\right].
\end{eqnarray}
\end{thm}

The first inequality in (\ref{E:Tracein}) is known; see e.g.,
\cite{bg}. The second inequality follows from the Lieb-Thirring
inequality \cite{lt}, and this has been recorded in the papers
\cite{fl,mds,wwy}. The last inequality follows from Theorem
\ref{T:Ineq} (i) upon choosing $X=I.$

Our next proposition gives a formula for the derivative, with
respect to $t,$ of $F_{t}(A,B).$ This result has been obtained
earlier as Proposition 15 in \cite{mds} and as the main ingredient
in the proof of Proposition 11 in \cite{wwy}. Our proof is
different.

\begin{prop} Let $A,B$ be positive definite matrices and let
$\varphi:{\Bbb R}\to {\Bbb P}$ be the function
$$\varphi(t)=A^{\frac{1-t}{2t}}BA^{\frac{1-t}{2t}}.$$
Let $F:{\Bbb R}_{+}\to {\Bbb R}_{+}$ be the function
$$F(t)={\mathrm{tr}} \varphi(t)^{t}=F_{t}(A,B).$$
Then
\begin{eqnarray}\label{E:dev}
F'(t)={\mathrm{tr}}\left[\varphi(t)^{t}\left(\log
\varphi(t)-\frac{1}{t}\log A\right)\right].
\end{eqnarray}
In particular,
\begin{eqnarray}\label{E:dev1}
F'(1)={\mathrm{tr}}\left[B(\log B-\log A)\right].
\end{eqnarray}
\end{prop}

\begin{proof}
We have
$$\varphi(t)=A^{\frac{1}{2t}}\left(A^{-\frac{1}{2}}BA^{-\frac{1}{2}}\right)A^{\frac{1}{2t}}.$$
Differentiation gives
$$\varphi'(t)=-\frac{1}{2t^{2}}\left((\log
A)\varphi(t)+\varphi(t)\log A\right).$$ Let $h:{\Bbb R}_{+}\to {\Bbb
H}$ be the map $h(t)=t\log \varphi(t).$ Then
\begin{eqnarray*}
h'(t)&=& \log \varphi(t)+t\varphi(t)^{-1}\varphi'(t)\\
&=&\log \varphi(t)-\frac{1}{2t}\varphi(t)^{-1}\left[(\log
A)\varphi(t)+\varphi(t)\log A\right].
\end{eqnarray*}
Our function $F(t)={\mathrm{tr}} e^{h(t)}.$ Hence
\begin{eqnarray*}
F'(t)&=&{\mathrm{tr}}\left(e^{h(t)}h'(t)\right)\\
&=&{\mathrm{tr}}\left(\varphi(t)\log\varphi(t)\right)-
\frac{1}{2t}{\mathrm{tr}}\left[\varphi(t)^{t-1}\left((\log
A)\varphi(t)+\varphi(t)\log
A\right)\right]\\
&=&{\mathrm{tr}}\left(\varphi(t)^{t}\log
\varphi(t)\right)-\frac{1}{t}{\mathrm{tr}}\left(\varphi(t)^{t}\log
A\right).
\end{eqnarray*}
This proves (\ref{E:dev}).
\end{proof}

\vspace{4mm}

\vspace{4mm} Using L'Hopital's rule and (\ref{E:dev1}) we obtain the
relation (\ref{E:RE}).

\section{Higher derivatives and strong convexity}
We now turn to the proof of Theorem \ref{T:main}. For $0<t<1,$ let
$\mu$ be the measure on $(0,\infty)$ defined by
$$d\mu(\lambda)=\frac{\sin t\pi}{\pi}\lambda^{t-1}d\lambda.$$
Then for all $x>0$ we have
\begin{eqnarray}\label{E:intre}
x^{t-1}=\int_{0}^{\infty}\frac{1}{\lambda+x}d\mu(\lambda).
\end{eqnarray}
See (V.4) in \cite{rbh}. Differentiating both sides with respect to
$x,$ we obtain
\begin{eqnarray}\label{E:intrep}
(1-t)x^{t-2}=\int_{0}^{\infty}\frac{1}{(\lambda+x)^{2}}d\mu(\lambda).
\end{eqnarray}
Let $h:{\Bbb P}\to {\Bbb R}$ be the function
$$h(X)=-\frac{1}{t}{\mathrm{tr}} X^{t}.$$ By Lemma \ref{L:diff}, the
derivative of $h$ is given by
\begin{eqnarray}
Dh(X)(Y)=-{\mathrm{tr}}X^{t-1}Y.
\end{eqnarray}
Let $g(X)=X^{t-1}.$ Then the second derivative $D^2h(X)$ is  the
symmetric bilinear function
$$D^2h(X)(Y,Z)=-{\mathrm{tr}}\left(Dg(X)(Z)\right)Y.$$
Using the integral representation (\ref{E:intre}) we see that
$$Dg(X)(Z)=-\int_{0}^{\infty}(\lambda+X)^{-1}Z(\lambda+X)^{-1}d\mu(\lambda),$$
and hence
\begin{eqnarray}\label{E:Hess}
D^{2}h(X)(Y,Z)={\mathrm{tr}}\int_{0}^{\infty}(\lambda+X)^{-1}Z(\lambda+X)^{-1}Y\
 d\mu(\lambda).
\end{eqnarray}
In the notation of gradients
$$D^{2}h(X)(Y,Z)=\langle \nabla^{2}h(X)(Y),Z\rangle.$$
So, we can write (\ref{E:Hess}) also as
\begin{eqnarray}
\nabla^{2}h(X)(Y)=\int_{0}^{\infty}(\lambda+X)^{-1}Y(\lambda+X)^{-1}\
 d\mu(\lambda).
\end{eqnarray}
In passing, we note that this shows $\nabla^{2}h(X)$ is a completely
positive linear map on the space ${\Bbb H}$ of Hermitian matrices.

Now let $A$ be any positive matrix and let
\begin{eqnarray}
{\tilde
h}(X)=h(A^{\frac{1}{2}}XA^{\frac{1}{2}})=-\frac{1}{t}{\mathrm{tr}}\left(A^{\frac{1}{2}}XA^{\frac{1}{2}}\right)^{t}.
\end{eqnarray}
Then
$$D{\tilde
h}(X)(Y)=Dh(A^{\frac{1}{2}}XA^{\frac{1}{2}})(A^{\frac{1}{2}}YA^{\frac{1}{2}})$$
and
$$D^{2}{\tilde
h}(X)(Y,Z)=D^{2}h(A^{\frac{1}{2}}XA^{\frac{1}{2}})(A^{\frac{1}{2}}YA^{\frac{1}{2}},A^{\frac{1}{2}}ZA^{\frac{1}{2}}).$$
Hence, from (\ref{E:Hess})
$$
D^2{\tilde
h}(X)(Y,Z)={\mathrm{tr}}\int_{0}^{\infty}\left(\lambda+A^{\frac{1}{2}}XA^{\frac{1}{2}}\right)^{-1}
A^{\frac{1}{2}}ZA^{\frac{1}{2}}\left(\lambda+A^{\frac{1}{2}}XA^{\frac{1}{2}}\right)^{-1}A^{\frac{1}{2}}YA^{\frac{1}{2}}d\mu(\lambda).$$
Using the identity
$$\left(\lambda+A^{\frac{1}{2}}XA^{\frac{1}{2}}\right)^{-1}=\left(A^{1/2}(\lambda A^{-1}+X)A^{1/2}\right)^{-1}=A^{-\frac{1}{2}}\left(\lambda
A^{-1}+X\right)^{-1}A^{-\frac{1}{2}},$$ we obtain
$$D^{2}{\tilde
h}(X)(Y,Z)={\mathrm{tr}}\int_{0}^{\infty}\left(\lambda
A^{-1}+X\right)^{-1}Y\left(\lambda A^{-1}+X\right)^{-1}Z\, d\mu(\lambda).$$
In other words
\begin{eqnarray}\label{E:exp}
\nabla^{2}{\tilde h}(X)(Y)=\int_{0}^{\infty}\left(\lambda
A^{-1}+X\right)^{-1}Y\left(\lambda A^{-1}+X\right)^{-1}
d\mu(\lambda).
\end{eqnarray}
Let $C_{\lambda}=\left(\lambda A^{-1}+X\right)^{-1}$ and let
$\Gamma_{C_{\lambda}}$ be the map on the space of matrices defined
as $\Gamma_{C_{\lambda}}(Y)=C_{\lambda}YC_{\lambda}.$ The
eigenvalues of $\Gamma_{C_{\lambda}}$ are the products of the
eigenvalues of $C_{\lambda}.$ The expression (\ref{E:exp}) can be
rewritten as
$$\nabla^{2}{\tilde
h}(X)(Y)=\int_{0}^{\infty}\Gamma_{C_{\lambda}}(Y) d\mu(\lambda).$$
By the extremal principle for eigenvalues
$$\frac{\langle \Gamma_{C_{\lambda}}(Y),Y\rangle}{\langle
Y,Y\rangle}\geq
\lambda_{\min}(\Gamma_{C_{\lambda}})=\lambda_{\min}(C_{\lambda})^{2}.$$
Now let $\alpha$ and $\beta$ be positive reals with $\alpha\leq
\beta$ and suppose that $\alpha I\leq X\leq \beta I.$ Then for all
$A\in {\Bbb P}$
$$\left(\frac{\lambda}{\lambda_{\min}(A)}+\beta\right)^{-1}\leq
C_{\lambda}\leq
\left(\frac{\lambda}{\lambda_{\max}(A)}+\alpha\right)^{-1}.$$ Using
the last three relations above, we get
\begin{eqnarray*}
\frac{\langle\nabla^{2}{\tilde h}(X)(Y),Y\rangle}{\langle
Y,Y\rangle}&\geq&\int_{0}^{\infty}
\left(\frac{\lambda}{\lambda_{\min}(A)}+\beta\right)^{-2}\
d\mu(\lambda)\\
&=&\lambda_{\min}(A)^{2}\int_{0}^{\infty}\frac{1}{(\lambda+\beta\lambda_{\min}(A))^{2}}\
d\mu(\lambda)\\
&=&(1-t)\beta^{t-2}\lambda_{\min}(A)^{t},
\end{eqnarray*}
the last equality being a consequence of (\ref{E:intrep}). This
shows that \begin{eqnarray}\label{E:SE} \nabla^2{\tilde h}(X)\geq
(1-t)\beta^{t-2}\lambda_{\min}(A)^{t},
\end{eqnarray}
for all $A\in {\Bbb P}$ and $\alpha I\leq X\leq \beta I.$

Finally, let $f$ be the function defined by (\ref{E:mainf}). Then
$-f$ is the function obtained from ${\tilde h}$ by multiplying it by
$t$ and replacing $A$ by $A^{\frac{1-t}{t}}.$ Hence, (\ref{E:SE})
leads to the inequality \begin{eqnarray}\label{ineq}
-\nabla^{2}f(X)\geq t(1-t)\beta^{t-2}\lambda_{\min}(A)^{1-t},
\end{eqnarray}
for all $\alpha I\leq X\leq \beta I.$ So, if we assume
$\lambda_{\min}(A)\geq \alpha,$ then we obtain the first inequality
in (\ref{E:convexity}).

The second inequality in (\ref{E:convexity}) has an analogous proof.\hfill\qed

\vspace{4mm}

Our method can be used to calculate higher derivatives of any order,
and to estimate their norms. For example, we can show that
$$\Vert\nabla^{3}f(X)\Vert\leq t(1-t)(2-t)\beta^{1-t}\alpha^{t-3},$$
from which it follows that
$$\Vert\nabla^{2}f(X)-\nabla^{2}f(Y)\Vert_{2}\leq
t(1-t)(2-t)\beta^{1-t}\alpha^{t-3}\Vert X-Y\Vert_{2}.$$

\section{Gradient
Projection Algorithm} Let $A_1,\ldots,A_m\in\mathbb{P}.$ For $0<t<1$ define the function $\varphi_t$ on $\mathbb{P}$ as
$$\varphi_t(X)=\sum_{j=1}^{m}w_{j}\left[{\mathrm{tr}}((1-t)A_{j}+tX)
-{\mathrm{tr}}\left(A_{j}^{\frac{1-t}{2t}}XA_{j}^{\frac{1-t}{2t}}\right)^{t}\right]$$
We consider the  optimization problem
\begin{eqnarray}\label{E:miniz}
{\underset{X\in\mathbb{P}}{\min}}\ \varphi_t(X)
\end{eqnarray}
on the convex cone $\mathbb{P}.$
The multimarginal optimal transport problem of Gaussian measures (\cite{ac,gm,gs}) is
the special case $t=1/2.$ Let $\alpha$ and $\beta$ be positive
numbers such that
 $$\alpha I\leq
A_{j}\leq \beta I, \ \ \ j=1,\dots,m.$$ We note that the optimal
values of $\alpha$ and $\beta$ are
$$\underset{1\leq j\leq m}{\min}\,\,\,\lambda_{\min}(A_{j})\ \ \ {\mathrm{and}} \ \ \ \underset{1\leq j\leq m}{\max}\,\,\,\lambda_{\max}(A_{j})$$
respectively. By the results obtained in the previous section, we see
that $\varphi_{t}$ is $t(1-t)\beta^{1-t}\alpha^{t-2}$-smooth and
$t(1-t)\beta^{t-2}\alpha^{1-t}$-strongly convex, and the condition
number of $\nabla^2\varphi_{t}(X)$ is bounded by
$\left(\frac{\beta}{\alpha}\right)^{3-2t}.$ By Proposition
\ref{P:Der} $\varphi_t$ is strictly convex with
\begin{eqnarray*}
D\varphi_t(X)(Y)&=&t\ \sum_{j=1}^{m}w_{j}
{\mathrm{tr}}\left[I-(A_{j}^{\frac{1-t}{t}}\#_{1-t}X^{-1})\right]Y\\
&=&t\ {\mathrm{tr}}\left[I-\sum_{j=1}^{m}w_{j}
(A_{j}^{\frac{1-t}{t}}\#_{1-t}X^{-1})\right]Y.
\end{eqnarray*}
In terms of the gradient
$$\nabla \varphi_t(X)=t\ \left[I-\sum_{j=1}^{m}w_{j}
(A_{j}^{\frac{1-t}{t}}\#_{1-t}X^{-1})\right].$$

 To prove the
existence and uniqueness of the minimization problem (\ref{E:miniz}), it
is enough to show that the equation $\nabla \varphi_t(X)=0$ has a
positive definite solution. This is equivalent to the nonlinear
matrix equation
$$X=\sum_{j=1}^{m}w_{j}X^{1/2}
\left(X^{-1}\#_{t}A_{j}^{\frac{1-t}{t}}\right)X^{1/2}=\sum_{j=1}^{m}w_{j}
\left(X^{1/2}A_{j}^{\frac{1-t}{t}}X^{1/2}\right)^{t}.$$ Let
$F:{\mathbb P}\to {\mathbb P}$ be the map defined by
$$F(X)=\sum_{j=1}^{m}w_{j}
\left(X^{1/2}A_{j}^{\frac{1-t}{t}}X^{1/2}\right)^{t}.$$
If all $A_j,$ $1\le j\le m,$ and $X$ are bounded from below by $\alpha I$ and from above by $\beta I,$
then
\begin{eqnarray*}
X^{1/2}A_{j}^{\frac{1-t}{t}}X^{1/2}&\leq&
X^{1/2}(\beta^{\frac{1-t}{t}}I)X^{1/2}\leq \beta^{\frac{1-t}{t}}
X\leq \beta^{\frac{1-t}{t}}\beta I\leq \beta^{1/t}I
\end{eqnarray*}
and hence $F(X)\leq \sum_{j=1}^{m}w_{j}\beta I=\beta I.$ Similarly
$F(X)\geq \alpha I.$ This shows that $F$ is a self-map on the
compact and convex interval $[\alpha I, \beta I]:=\{X>0: \alpha
I\leq X\leq \beta\}.$ By Brouwer's fixed point theorem, $F$ has a
fixed point. This settles the problem of  existence and uniqueness
of the minimizer in \eqref{E:miniz}.

Now we apply the classical gradient projection method for
(constrained) strongly convex functions. Let
\begin{eqnarray*}
X_{k+1}&=&\left[X_{k}-\eta \nabla f(X_{k})\right]_{+}\\
&=&\left[X_{k}-t\eta
I+t\eta\sum_{j=1}^{n}w_{j}\left(A_{j}^{\frac{1-t}{t}}\#_{1-t}X_{k}^{-1}\right)\right]_{+}
\end{eqnarray*}
where $X_{0}\in [\alpha I, \beta I]$ and $ [\cdot]_{+}$ denotes the
projection to $[\alpha I, \beta I]$ and
$0<\eta<\frac{2}{\beta_{*}}.$ Since
 $\varphi_{t}$ is
$\beta_{*}:=t(1-t)\beta^{1-t}\alpha^{t-2}$-smooth and
$\alpha_{*}:=t(1-t)\beta^{t-2}\alpha^{1-t}$-strongly convex, the
iteration  converges to the unique minimizer $X_{*}$ with linear
convergence rate
\begin{eqnarray}\label{E:cr}\Vert X_{k+1}-X_{*}\Vert_{2}\leq
q^{k}\Vert X_{k}-X_{*}\Vert_{2} \end{eqnarray} where
$$q=\max\{|1-\eta\alpha_{*}|, |1-\eta\beta_{*}|\}.$$
Or, with $\eta=1/\beta_{*},$
$$\Vert X_{k+1}-X_{*}\Vert^{2}_{2}\leq e^{-\frac{k\alpha_{*}}{\beta_{*}}}\Vert X_{1}-X_{*}\Vert_{2}
=e^{-k\left(\alpha/\beta\right)^{3-2t}}\Vert
X_{1}-X_{*}\Vert^2_{2}.$$ See (Theorem 3.10, \cite{bu}). A
gradient-based optimization method with sublinear convergence for
$t=1/2$ has recently appeared in \cite{kly}.

\section{Appendix}

Some of the inequalities in \eqref{E:Tracein} have much stronger
versions, and these are related to recurring themes in matrix
analysis and mathematical physics.  See e.g.,
\cite{an,rbh,rbh1,bg,c,lt}.
 \vskip.1in Let $x,y$ be two $n$-vectors
with nonnegative components. Let $x_1^\downarrow\ge\cdots\ge
x_n^\downarrow$ be the components of $x$ arranged in decreasing
order. We say that $x$ is {\it weakly log majorised} by $y,$ in
symbols $x\prec_{\textrm{wlog}} y,$ if for $1\le k\le n$
\begin{equation}
\prod\limits_{j=1}^{k}x_j^\downarrow\le\prod\limits_{j=1}^{k}y_j^\downarrow.\label{eqap1}
\end{equation}

If in addition
$$\prod\limits_{j=1}^{n}x_j^\downarrow=\prod\limits_{j=1}^{n}y_j^\downarrow,$$
then we say that $x$ is {\it log majorised} by $y,$
and write this as $x\prec_{\log}y.$
We write $x\le y$ if $x_j^\downarrow\le y_j^\downarrow$ for all $j=1,\ldots,n.$
\vskip.1in
Let $X$ be any $n\times n$ matrix and let
$\lambda(X)=(\lambda_1(X),\ldots,\lambda_n(X))$ and $s(X)=(s_1(X),\ldots,s_n(X))$ be the $n$-tuples
whose components are the eigenvalues and the singular values of $X,$ respectively.
A famous inequality of H. Weyl says that
\begin{equation}
(|\lambda_1(X)|,\ldots,|\lambda_n(X)|)\prec_{\log}s(X).\label{eqap2}
\end{equation}
(See \cite{rbh} p. 43.)
\vskip.1in
Now let $A$ and $B$ be positive definite matrices and let $0< t< 1.$
It has been shown in \cite{bg} that
\begin{equation}
\lambda(A\#_tB)\prec_{\log}\lambda(A^{1-t}B^t).\label{eqap3}
\end{equation}
A matrix version of Young's inequality proved by T. Ando \cite{an} says that
\begin{equation}
s(A^{1-t}B^t)\le \lambda((1-t)A+tB).\label{eqap4}
\end{equation}
Combining \eqref{eqap3} and \eqref{eqap4}
we have the chain
\begin{eqnarray}
\lambda(A\#_tB) &\prec_{\log}& \lambda(A^{1-t}B^t)\prec_{\log}s(A^{1-t}B^t)\nonumber\\
&\le & \lambda((1-t)A+tB).\label{eqap5}
\end{eqnarray}
The inequality \eqref{E:Tracein} raises the intriguing question of how the eigenvalue tuple
$\lambda\left(A^{\frac{1-t}{2t}}BA^{\frac{1-t}{2t}}\right)^t$ fits into this chain.
To answer this we recall the Araki-Lieb-Thirring inequalities
which say that if $X$ and $Y$ are positive definite matrices,
then
\begin{equation}
\lambda(X^tY^tX^t)\prec_{\log}\lambda(XYX)^t,\textrm{ for }0\le t\le 1,\label{eqap6}
\end{equation}
and
\begin{equation}
\lambda(XYX)^t\prec_{\log}\lambda(X^tY^tX^t),\textrm{ for }t\ge 1.\label{eqap7}
\end{equation}
See the proof of Theorem IX.2.10 in \cite{rbh}.
Using the first of these inequalities,
we see that for $0\le t\le 1,$
\begin{equation}
\lambda(A^{1-t}B^t)=\lambda\left(A^{\frac{1-t}{2}}B^tA^{\frac{1-t}{2}}\right)\prec_{\log}\lambda\left(A^{\frac{1-t}{2t}}BA^{\frac{1-t}{2t}}\right)^t.\label{eqap8}
\end{equation}

Now suppose $\frac{1}{2}\le t\le 1.$
Then from \eqref{eqap7} we obtain
$$\lambda\left(A^{\frac{1-t}{2t}}BA^{\frac{1-t}{2t}}\right)^{2t}\prec_{\log}\lambda\left(A^{1-t}B^{2t}A^{1-t}\right).$$
Taking square roots of both sides, we get
\begin{equation}
\lambda\left(A^{\frac{1-t}{2t}}BA^{\frac{1-t}{2t}}\right)^t\prec_{\log} s(A^{1-t}B^t).\label{eqap9}
\end{equation}
Combining \eqref{eqap5}, \eqref{eqap8} and \eqref{eqap9} we have
\begin{eqnarray}
\lambda(A\#_tB) &\prec_{\log}& \lambda(A^{1-t}B^t)\nonumber\\
&\prec_{\log}& \lambda\left(A^{\frac{1-t}{2t}}BA^{\frac{1-t}{2t}}\right)^t\nonumber\\
&\prec_{\log}& s(A^{1-t}B^t)\nonumber\\
&\le & \lambda((1-t)A+tB),\label{eqap10}
\end{eqnarray}
for $\frac{1}{2}\le t\le 1.$

On the other hand if $0\le t\le \frac{1}{2},$
then from \eqref{eqap6} we obtain
$$\lambda\left(A^{1-t}B^{2t}A^{1-t}\right)\prec_{\log}\lambda\left(A^{\frac{1-t}{2t}}BA^{\frac{1-t}{2t}}\right)^{2t}.$$
Taking square roots of both sides we get
$$s(A^{1-t}B^t)\prec_{\log}\lambda\left(A^{\frac{1-t}{2t}}BA^{\frac{1-t}{2t}}\right)^t.$$
So for $0\le t\le\frac{1}{2},$ we have
\begin{eqnarray}
\lambda(A\#_tB) &\prec_{\log}& \lambda(A^{1-t}B^t)\nonumber\\
&\prec_{\log}& s(A^{1-t}B^t)\nonumber\\
&\prec_{\log}& \lambda\left(A^{\frac{1-t}{2t}}BA^{\frac{1-t}{2t}}\right)^t.\label{eqap11}
\end{eqnarray}
To complete this chain in the same way as \eqref{eqap10} it remains to
answer whether for $0\le t\le \frac{1}{2},$
$\lambda\left(A^{\frac{1-t}{2t}}BA^{\frac{1-t}{2t}}\right)^t$ is dominated by $\lambda((1-t)A+tB).$

 \vskip.2in \noindent{\it Acknowledgement.} We thank the anonymous referee and Professor Fumio Hiai
  for a careful reading of the manuscript. The work of R.
Bhatia is supported by a J. C. Bose National Fellowship and of Y.
Lim is supported by the National Research Foundation
 of Korea (NRF) grant founded by the Korea government (MEST) (No. 2015R1A3A2031159) and 2016R1A5A1008055.

\vskip.2in


\end{document}